\documentclass{amsart}

\usepackage{amsfonts}
\usepackage{amsmath}
\usepackage{amssymb}
\usepackage{amsthm}

\makeatletter
\renewcommand\@seccntformat[1]{\csname the#1\endcsname.\quad}

\newtheorem{theorem}{Theorem}[section]

\newtheorem{corollary}[theorem]{Corollary}

\newtheorem{proposition}[theorem]{Proposition}
\newtheorem{remark}[theorem]{Remark}

\newcommand{\norm}[3]{\ensuremath{\left\Vert#1\right\Vert_{#2}^{#3}}}
\newcommand{\abs}[3]{\ensuremath{\left\vert#1\right\vert_{#2}^{#3}}}
\DeclareMathOperator{\supp}{supp}

\begin{document}

\author[M. Ciesielski and G. Lewicki]{Maciej Ciesielski and Grzegorz Lewicki}

\title[ Substochastic operators in symmetric spaces]{Substochastic operators in symmetric spaces}

\begin{abstract}
First, we solve a crucial problem under which conditions increasing uniform $K$-monotonicity is equivalent to lower locally uniform $K$-monotonicity. Next, we investigate properties of substochastic operators on $L^1+L^\infty$ with applications. Namely, we show that a countable infinite combination of substochastic operators is also substochastic. Using $K$-monotonicity properties, we prove several theorems devoted to the convergence of the sequence of substochastic operators in the norm of a symmetric space $E$ under addition assumption on $E$. In our final discussion we focus on compactness of admissible operators for arbitrary Banach couples.
\end{abstract}

\maketitle


{\small \underline{2000 Mathematics Subjects Classification: 46E30, 46B20, 46B42.}\hspace{1.5cm}\
\ \ \quad\ \quad  }\smallskip\ 

{\small \underline{Key Words and Phrases:}\hspace{0.15in} Symmetric spaces, Banach couples, substochastic operators, admissible operators, uniform $K$-monotonicity.}


\section{Introduction}
Substochastic operators have been researched and applied intensively in many fields of mathematics since the early 20th century. Namely, the first interesting result concerning a substochastic operator appeared in 1929 (see \cite{HrLiPo29,MarOlkArn}), where Hardy, Littlewood and P\'olya presented a connection between existence of double stochastic matrix and a majorization one vector by another with respect to the Hardy-Littlewood-P\'olya relation $\prec$. Next, in 60-ties Calder\'on and Ryff showed that for any two non-negative measurable functions $f,g$ such that $f\prec{g}$ and $f\in{L^1+L^\infty}$ there exists a substochastic operator $T$ such that $Tg=f$ a.e. (see \cite{Ryff65,Cald66}). It is worth mentioning that this result was a crucial key in the Calder\'on's proof to establish an equivalence criterion for the exact interpolation space between a Banach couple $(L^1,L^\infty)$ \cite{Cald66}. It is necessary to recall some investigations devoted to relationships between substochastic operators and semigroups, among other used to explore properties of Markov processes or Markov chains, which have a substantial history dating back also to the 1960s (see \cite{Hor66,Jai66,Jam67}). The next motivation for this paper we find in \cite{ChDSS,Cies-UKM-Rot,Cies-LLUKM,CieLewUKM}, where authors presented equivalent conditions for $K$-monotonicity properties with applications in symmetric spaces. In the spirit of the previous investigations we give a complete answer for the essential problem about a relation between uniform $K$-monotonicity properties. Namely, we provide complete criteria for equivalence of increasing uniform $K$-monotonicity and lower local uniform $K$-monotonicity in symmetric spaces.
The second goal of this paper is to study a relationship between convergence of substochastic operators and $K$-monotonicity properties in symmetric spaces. First, we prove that a countable infinite combination of substochastic operators is substochastic and discuss some basic examples of 
substochastic operators in symmetric spaces. We start our investigation assuming the weakest property lower locally uniform $K$-monotonicity in a symmetric space $E$. Under additional assumption of the symmetric space $E$, we establish a norm convergence of a sequence of specific substochastic operators given by a countable infinite sum of the integral averaging on positive finite measurable sets. Next, letting the stronger property decreasing uniform $K$-monotonicity in a symmetric space $E$ we research norm and point convergence of the sequence of substochastic operators given in the descending order with respect to the Hardy-Littlewood-P\'olya relation $\prec$. In view of the previous result we also research norm convergence of $A^n$ a geometric sequence generated by a fixed substochastic operator $A$ in a symmetric space $E$. We also focus our attention on the convergence of a sequence of substochastic operator given in the ascending order for the partial order $\leq$ in a symmetric space $E$, equipped with uniformly $K$-monotone norm.
Finally, we show compactness of the admissible operator between Banach couples under additional assumptions of the interpolation spaces.   
 
\section{Preliminaries}

Now we recall basic notation and definition. Let $\mathbb{R}$, $\mathbb{R}^+$ and $\mathbb{N}$ be the sets of reals, nonnegative reals and positive integers, respectively. In a Banach space $(X,\norm{\cdot}{X}{})$ we denote the unit sphere and the closed unit ball by $S(X)$ and $B(X))$, respectively. A function $\phi:\mathbb{R}^+\rightarrow\mathbb{R}^+$ is called \textit{quasiconcave} if $\phi(t)$ is increasing and $\phi(t)/t$ is decreasing on $\mathbb{R}^+$ and also $\phi(t)=0\Leftrightarrow{t=0}$. We denote as usual by $\mu$ the Lebesgue measure on $I=[0,\alpha)$, where $\alpha =1$ or $\alpha =\infty$, and by $L^{0}$ the set of all (equivalence classes of) extended real valued Lebesgue measurable functions on $I$. Let us use the notation $A^{c}=I\backslash A$ for any measurable set $A$. By a \textit{Banach function space} (or a \textit{K\"othe space}) we mean a Banach lattice $(E,\Vert \cdot \Vert _{E})$ that is a sublattice of $L^{0}$ and possesses the following conditions
 
\begin{itemize}
\item[(1)] If $x\in L^0$, $y\in E$ and $|x|\leq|y|$ a.e., then $x\in E$ and $%
\|x\|_E\leq\|y\|_E$.
\item[(2)] There exists a strictly positive $x\in E$.
\end{itemize}
To simplify the notation in the whole paper we use the symbol $E^{+}={\{x \in E:x \ge 0\}}$. An element $x\in E$ is called a \textit{point of order continuity} if for any $(x_{n})\subset{}E^+$ such that $x_{n}\leq \left\vert x\right\vert 
$ and $x_{n}\rightarrow 0$ a.e. we get $\left\Vert x_{n}\right\Vert
_{E}\rightarrow 0.$ A Banach function space $E$ is said to be \textit{order continuous 
}(shortly $E\in \left( OC\right) $) if any element $x\in{}E$ is a point of order continuity (see \cite{LinTza}). We say that a Banach function space $E$ has the \textit{Fatou property} if for any $\left( x_{n}\right)\subset{}E^+$, $\sup_{n\in \mathbb{N}}\Vert x_{n}\Vert
_{E}<\infty$ and $x_{n}\uparrow x\in L^{0}$, then we have $x\in E$ and $\Vert x_{n}\Vert _{E}\uparrow\Vert x\Vert
_{E}$. Unless it is said otherwise, in the whole article it is considered that $E$ has the Fatou property. For any function $x\in L^{0}$ we denote the \textit{distribution function} by 
\begin{equation*}
d_{x}(\lambda) =\mu\left\{ s\in [ 0,\alpha) :\left\vert x\left(s\right) \right\vert >\lambda \right\},\qquad\lambda \geq 0.
\end{equation*}
and also the \textit{decreasing rearrangement} of $x$ by 
\begin{equation*}
x^{\ast }\left( t\right) =\inf \left\{ \lambda >0:d_{x}\left( \lambda
\right) \leq t\right\}, \text{ \ \ } t\geq 0.
\end{equation*}
The following notation is used $x^{*}(\infty)=\lim_{t\rightarrow\infty}x^{*}(t)$ if $\alpha=\infty$ and $x^*(\infty)=0$ if $\alpha=1$. For any function $x\in L^{0}$ we define the \textit{maximal function} of $x^{\ast }$ by 
\begin{equation*}
x^{\ast \ast }(t)=\frac{1}{t}\int_{0}^{t}x^{\ast }(s)ds.
\end{equation*}
For any element $x\in L^{0}$, it is well known that $x^{\ast }\leq x^{\ast \ast },$ $x^{\ast \ast }$ is decreasing, continuous and subadditive. For more details of $d_{x}$, $x^{\ast }$ and $x^{\ast \ast }$ the reader is referred to see \cite{BS, KPS}. Let us mention that two functions $x,y\in{L^0}$ are said to be \textit{equimeasurable} (shortly $x\sim y$) if $d_x=d_y$. A Banach function space $(E,\Vert \cdot \Vert_{E}) $ is called \textit{symmetric} or \textit{rearrangement invariant} (r.i. for short) if whenever $x\in L^{0}$ and $y\in E$ such that $x \sim y,$ then $x\in E$ and $\Vert x\Vert_{E}=\Vert y\Vert _{E}$. For a symmetric space $E$ we define $\phi_{E}$ its \textit{fundamental function} given by $\phi_{E}(t)=\Vert\chi_{(0,t)}\Vert_{E}$ for any $t\in [0,\alpha)$ (see \cite{BS}). For any two functions $x,y\in{}L^{1}+L^{\infty }$ the \textit{Hardy-Littlewood-P\'olya relation} $\prec$ is given by 
\begin{equation*}
x\prec y\Leftrightarrow x^{\ast \ast }(t)\leq y^{\ast \ast }(t)\text{ for
all }t>0.\text{ }
\end{equation*}
A symmetric space $E$ is called $K$-\textit{monotone} (shortly $E\in(KM)$)
if for any $x\in L^{1}+L^{\infty}$ and $y\in E$ with $x\prec y,$ then $x\in E$ and $\Vert x\Vert_{E}\leq \Vert y\Vert _{E}.$ It is commonly known that a symmetric space $E$ is $K$-monotone if and only if $E$ is exact interpolation space between $L^{1}$ and $L^{\infty }.$ It is worth mentioning that a symmetric space $E$ equipped with an order continuous norm or with the Fatou property is $K$-monotone (see \cite{KPS}).
Given $x\in{E}$ is called a \textit{point of lower $K$-monotonicity} of $E$ for short an $LKM$ \textit{point} (resp. a \textit{point of upper $K$-monotonicity} of $E$ shortly an $UKM$ \textit{point}) if for any $y\in{E}$, $y^*\neq{x^*}$ with $y\prec{x}$ (resp. $x\prec{y}$), then $\norm{y}{E}{}<\norm{x}{E}{}$ (resp. $\norm{x}{E}{}<\norm{y}{E}{}$). Let us mention that given a symmetric space $E$ is called \textit{strictly $K$-monotone} (shortly $E\in(SKM)$) if any element of $E$ is an $LKM$ point or equivalently any element of $E$ is a $UKM$ point.

 Given $x\in{E}$ is said to be a \textit{point of lower local uniform $K$-monotonicity} of $E$, shortly an $LLUKM$ point, if for any $(x_n)\subset{E}$ with $x_n\prec{x}$ for all $n\in\mathbb{N}$ and $\norm{x_n}{E}{}\rightarrow\norm{x}{E}{}$, then $\norm{x^*-x_n^{*}}{E}{}\rightarrow{0}$.

 A symmetric space $E$ is called \textit{Lower locally uniformly $K$-monotone}, shortly $E\in(LLUKM)$, whenever any element $x\in{}E$ is a $LLUKM$ point.
 
 We refer the reader for more information to see \cite{ChDSS,Cies-JAT,Cies-geom,CieKolPluSKM}.
In \cite{CieLewUKM}, authors introduced some new notions which are a generalization of uniform monotonicity properties in symmetric spaces. The generalization was obtained by replacing a partial order $\leq$ by a weaker relation $\prec$, in definition of monotonicity properties. It is worth mentioning that the new properties characterizes a completely different geometric structure of symmetric spaces than monotonicity properties. A symmetric space $E$ is said to be \textit{uniformly $K$-monotone}, shortly $E\in(UKM)$, if  for any $(x_n),(y_n)\subset{E}$ such that $x_n\prec{y_n}$ for all $n\in\mathbb{N}$ and $\lim_{n\rightarrow\infty}\norm{x_n}{E}{}=\lim_{n\rightarrow\infty}\norm{y_n}{E}{}<\infty$ we have $\norm{x_n^*-y_n^*}{E}{}\rightarrow{0}$.
A symmetric space $E$ is called \textit{decreasing (resp. increasing) uniformly $K$-monotone}, shortly $E\in(DUKM)$ (resp. shortly $E\in(IUKM)$), if  for any $(x_n),(y_n)\subset{E}$ such that $x_{n+1}\prec{}x_n\prec{y_n}$ for all $n\in\mathbb{N}$ and $\lim_{n\rightarrow\infty}\norm{x_n}{E}{}=\lim_{n\rightarrow\infty}\norm{y_n}{E}{}<\infty$ (resp. $x_n\prec{y_n}\prec{y_{n+1}}$ for every $n\in\mathbb{N}$ and $\lim_{n\rightarrow\infty}\norm{x_n}{E}{}=\lim_{n\rightarrow\infty}\norm{y_n}{E}{}<\infty$), we have $\norm{x_n^*-y_n^*}{E}{}\rightarrow{0}$.
 
An operator $T$ from a Banach function space $(X,\norm{\cdot}{X}{})$ into a Banach function space $(Y,\norm{\cdot}{Y}{})$ is called a \textit{positive contraction} if its norm is at most $1$ and it satisfies the property that $T(x)\geq{0}$ whenever $x\geq{0}$.  An admissible operator for a Banach couple $(L^1,L^\infty)$ is said to be a \textit{substochastic operator} whenever it is positive contraction on $L^1$ and on $L^\infty$. In addition, on the space $L^1+L^\infty$ we may employ an equivalent characterization of substochastic operators, i.e. $T$ is substochastic if and only if $T(x)\prec{x}$ for all $x\in L^1+L^\infty$ (see \cite{BS}). 


Unless we say otherwise, we assume in the whole paper that two substochastic operators $T$ and $S$ satisfies $T\leq S$ (resp. $T\prec S$) on $L^1+L^\infty$, whenever for any $x\in{L^1+L^\infty}$ we have $Tx\leq Sx$ (resp. $Tx\prec Sx$). 	
For any quasiconcave function $\phi$ on $I$ the Marcinkiewicz function space $M_{\phi}^{(*)}$ (resp. $M_{\phi}$) is a subspace of $L^0$ such that
$$\norm{x}{M_{\phi}^{(*)}}{}=\sup_{t>0}\{x^*(t)\phi(t)\}<\infty$$
$${\left(\textnormal{resp. }\norm{x}{M_{\phi}}{}=\sup_{t>0}\{x^{**}(t)\phi(t)\}<\infty\right).}$$
It is well known that $M_{\phi}\overset{1}{\hookrightarrow}M_{\phi}^{(*)}$ and also the Marcinkiewicz space $M_{\phi}^{(*)}$ (resp. $M_{\phi}$) is a r.i. quasi-Banach function space (r.i. Banach function space) with the fundamental function $\phi$ on $I$. Let us recall the common fact that $E\overset{1}{\hookrightarrow}M_{\phi}$ for any symmetric space $E$ with the fundamental function $\phi$ (for more details see \cite{BS,KPS}). 

\section{Increasing uniform $K$-monotonicity}

\begin{theorem}\label{thm:full:descr:IUKM}
	Let $E$ be a symmetric space and let $\phi_E$ be the fundamental function of $E$. Then, $E$ is increasing uniformly $K$-monotone if and only if $E$ is lower locally uniformly $K$-monotone and $$\lim_{t\rightarrow{0^+}}\frac{t}{\phi_E(t)}=0.$$
\end{theorem}

\begin{proof}
	\textit{Necessity}. Let $(x_n),(y_n)\subset{E}$, $x_n\prec{y_n}\prec{y_{n+1}}$ and 
	\begin{equation}\label{equ:IUKM}
	\lim_{n\rightarrow\infty}\norm{y_n}{E}{}=\lim_{n\rightarrow\infty}\norm{x_n}{E}{}=d,
	\end{equation}
	where $d\in(0,\infty)$. Clearly, by Proposition 5.9 in \cite{BS} we get $E$ is embedded in the Marcinkiewicz space $M_{\phi_E}$ with the norm one, and so
	\begin{equation*}
	\phi_E(t)y_n^*(t)\leq\norm{y_n}{E}{}\leq{d}<\infty
	\end{equation*}
	for all $t>0$, and $n\in\mathbb{N}$. Hence, by Helly's Selection Theorem \cite{STV}, passing to subsequence if necessary we may assume that there is $y_0^*=y_0\in{L^0}$ such that $y_n^*\rightarrow{y_0^*}$ a.e. on $I$. Next, applying \eqref{equ:IUKM} and by assumption that $E$ has the Fatou property, in view of Fatou's Lemma \cite{BS} it follows that $y_0\in{E}$ and
	\begin{equation}\label{equ:1:IUKM}
	\norm{y_0}{E}{}\leq\liminf_{n\rightarrow\infty}\norm{y_n}{E}{}. 
	\end{equation}
	Moreover, we have for any $t>0$ and $n\in\mathbb{N}$,
	\begin{equation}\label{equ:0:IUKM}
	x_n^{**}(t)\leq{y}_n^{**}(t)\leq{y_{n+1}^{**}(t)}\leq\frac{\norm{y_n}{E}{}}{\phi_E(t)}\leq\frac{d}{\phi_E(t)}.
	\end{equation}
	Consequently, for any $t>0$ we obtain
	\begin{equation*}
	\liminf_{n\rightarrow\infty}{y}_n^{**}(t)=\lim_{n\rightarrow\infty}{y}_n^{**}(t)\leq\frac{d}{\phi_E(t)}.
	\end{equation*}
	Hence, since $y_n^*\rightarrow{y_0^*}$, by Fatou's Lemma we conclude that
	\begin{equation}\label{equ:2:IUKM}
	{y_0^{**}(t)}\leq\lim_{n\rightarrow\infty}{y}_n^{**}(t)<\infty
	\end{equation}
	for any $t\in{I}$, $t>0$. Now, in view of Corollary 5.3 in \cite{BS}, by assumption $\lim_{t\rightarrow{0^+}}t/\phi_{E}(t)=0$ and by \eqref{equ:0:IUKM} it follows that for any $\epsilon>0$ there exists $t_\epsilon>0$ such that for all $t\in(0,t_\epsilon]$ and $n\in\mathbb{N}$,
	\begin{equation}\label{equ:4:IUKM}
	\frac{1}{d}\int_{0}^{t}y_n^*\leq\frac{t}{\phi_E(t)}<\frac{\epsilon}{4}.
	\end{equation}
	Furthermore, since $y_n^*\rightarrow{y_0^*}$ a.e. on $I$, by the Lebesgue Dominated Convergence Theorem we observe that for any $0<s<t$, there exists $N_{\epsilon}\in\mathbb{N}$ such that for all $n\geq{}N_{\epsilon}$,
	\begin{equation*}
	\abs{\int_s^t{y_n^*}-\int_s^t{y_0^*}}{}{}<\frac{d\epsilon}{4}.
	\end{equation*}
	Next, without loss of generality taking additionally $s<t_\epsilon$ for the previous inequality and using \eqref{equ:4:IUKM} for any $n\geq{N_\epsilon}$ we get
	\begin{align*}
	\frac{1}{d}\int_{0}^{t}y_n^{*}=\frac{1}{d}\int_{0}^{s}y_n^{*}+\frac{1}{d}\int_{s}^{t}y_n^{*}&\leq\frac{1}{d}\int_{0}^{t_\epsilon}y_n^*+\frac{1}{d}\abs{\int_{s}^{t}y_n^*-\int_{s}^{t}y_0^*}{}{}+\frac{1}{d}\int_{s}^{t}y_0^*\\
	&<\frac{t_\epsilon}{\phi_E(t_\epsilon)}+\frac{\epsilon}{4}+\frac{1}{d}\int_{s}^{t}y_0^*<\frac{\epsilon}{2}+\frac{1}{d}\int_{0}^{t}y_0^*.
	\end{align*}
	In consequence, applying \eqref{equ:0:IUKM} and \eqref{equ:2:IUKM} we conclude that for all $t\in{I}$ with $t>0$,
	\begin{equation}\label{equ:3:IUKM}
	\lim_{n\rightarrow\infty}y_n^{**}(t)={y_0}^{**}(t).
	\end{equation}	
	Now, since $y_n\prec{y_{n+1}}$ for any $n\in\mathbb{N}$, we claim that $y_n\prec{y_0}$ for all $n\in\mathbb{N}$. Indeed, if it is not true then there exist $t_0>0$, $\epsilon_0>0$ and $(n_k)\subset\mathbb{N}$ such that for every $k\in\mathbb{N}$ we have 
	\begin{equation*}
	y_0^{**}(t_0)<y_{n_k}^{**}(t_0)-\epsilon_0.
	\end{equation*} 
	On the other hand, in view of \eqref{equ:3:IUKM} there exists $N_0\in\mathbb{N}$ such that for each $n\geq{N_0}$,
	\begin{equation*}
	y_n^{**}(t_0)<y_0^{**}(t_0)+\epsilon_0/2,
	\end{equation*} 
	which yields a contradiction and proves the claim. Hence, by assumption $x_n\prec{y_n}\prec{y_{n+1}}$ for all $n\in\mathbb{N}$ we get $x_n\prec{y_0}$ for any $n\in\mathbb{N}$. Therefore, by symmetry of $E$ we have $\norm{x_n}{E}{}\leq\norm{y_0}{E}{}$, and so in view of the facts \eqref{equ:IUKM} and \eqref{equ:1:IUKM} we observe that
	\begin{equation}\label{equ:6:IUKM}
	d=\lim_{n\rightarrow\infty}\norm{x_n}{E}{}=\lim_{n\rightarrow\infty}\norm{y_n}{E}{}=\norm{y_0}{E}{}. 
	\end{equation}
	In consequence, sine $x_n\prec{y_0}$ and $y_n\prec{y_0}$ for any $n\in\mathbb{N}$, by assumption that $E$ is lower locally uniformly $K$-monotone it follows that 
	\begin{equation*}
	\lim_{n\rightarrow\infty}\norm{x_n^*-y_0^*}{E}{}=\lim_{n\rightarrow\infty}\norm{y_n^*-y_0^*}{E}{}=0.
	\end{equation*} 
	Finally, using the triangle inequality of the norm in $E$ we obtain
	\begin{equation*}
	\lim_{n\rightarrow\infty}\norm{y_n^*-x_n^*}{E}{}=0.
	\end{equation*}
	\textit{Sufficiency.} Immediately, by Remark 4.1 in \cite{CieLewUKM} it is enough to prove that increasing uniform $K$-monotonicity implies
	$$\lim_{t\rightarrow{0^+}}\frac{t}{\phi_E(t)}=0.$$
	By Corollary 5.3 in \cite{BS}, we suppose for a contrary that $\lim_{t\rightarrow{0^+}}{\phi_E(t)/t}=c>0$. Define for any $n\in\mathbb{N}$,
	\begin{equation*}
		x_n=n\chi_{[0,\frac{1}{n})}\qquad\textnormal{and}\qquad y_n=2n\chi_{[0,\frac{1}{2n})}.
	\end{equation*}
	Then, it is easy to see that $x_n\prec{y_n}\prec{y_{n+1}}$ for all $n\in\mathbb{N}$ and 
	\begin{equation*}
	\lim_{n\rightarrow\infty}\norm{x_n}{E}{}=\lim_{n\rightarrow\infty}\frac{\phi_{E}(1/n)}{1/n}=c\qquad\textnormal{and}\qquad\lim_{n\rightarrow\infty}\norm{y_n}{E}{}=\lim_{n\rightarrow\infty}\frac{\phi_{E}(1/2n)}{1/2n}=c.
	\end{equation*}
	In consequence, by assumption that $E$ is increasing uniformly $K$-monotone it follows that 
	\begin{equation*}
	\lim_{n\rightarrow\infty}\norm{y_n^*-x_n^*}{E}{}=0.
	\end{equation*}
	On the other hand, by  symmetry of $E$ we have
	\begin{align*}
	\lim_{n\rightarrow\infty}\norm{y_n^*-x_n^*}{E}{}&=\lim_{n\rightarrow\infty}\norm{2n\chi_{[0,\frac{1}{2n})}-n\chi_{[0,\frac{1}{n})}}{E}{}=\lim_{n\rightarrow\infty}\norm{n\chi_{[0,\frac{1}{2n})}-n\chi_{[\frac{1}{2n},\frac{1}{n})}}{E}{}\\
	&=\lim_{n\rightarrow\infty}\frac{\phi_{E}(1/n)}{1/n}=c>0,
	\end{align*} 
	which contradicts and completes the proof.
\end{proof}

\section{Substochastic operators}

This section is devoted to investigation of convergence of the substochastic operators in symmetric spaces. First, we prove that any countable sum of substochastic operators is substochastic. 

\begin{remark}
	Let us recall some basic facts about substochastic operators, which will use in our further investigation. Letting $\mathbb{S}$ is a family of all substochastic operators on $L^1+L^\infty$ it is easy to show that $\mathbb{S}$ is convex and closed under the composition. Indeed, by subadditivity of the Hardy operator (see Theorem 3.4 in \cite{BS}) taking $A,B\in\mathbb{S}$, for any $\theta\in[0,1]$ and $x\in L^1+L^\infty$ we have
	\begin{align*}
		\theta A(x)+(1-\theta)B(x)\prec x\quad\textnormal{and}\quad A\circ B(x)\prec B(x)\prec x.
	\end{align*}
\end{remark}

\begin{proposition}\label{prop:sum:substoch:operat}
	Let $(T_j)$ be a sequence of substochastic operators on a Banach pair $(L^1,L^\infty)$ and let $(E_j)_{j\in\mathbb{J}}$ be a countable collection of pairwise disjoint subsets of $I$, each with finite positive measure. Assume that $E=I\setminus{\bigcup_{j\in\mathbb{J}}}E_j$ and $f\in{L^0}$ be locally integrable function on every $E_j$. Then the operator $\mathcal{A}:L^0\rightarrow{L^0}$ given by 
	\begin{equation*}
		\mathcal{A}(f)=f\chi_{E}+\sum_{j\in J}T_j(f\chi_{E_j})\chi_{E_j}
	\end{equation*}
	holds $\mathcal
	{A}(f)\prec{f}$.
\end{proposition}

\begin{proof}
	Assume first $\mathcal{A}(f)=f\chi_{E}+T_1(f\chi_{E_1})\chi_{E_1}$ for $0<\mu(E_1)<\infty$ and $E=I\setminus{E_1}$. Then, taking $F\subset{I}$, $t=\mu(F)\in{I}\setminus\{0\}$, $t_0=\mu(F \cap E_1)$ we observe that
	\begin{align*}
		\int_{F}|\mathcal{A}f|d\mu&=\int_{F\cap{E}}|\mathcal{A}f|d\mu+\int_{F\cap E_1}|\mathcal{A}f|d\mu\\
		&=\int_{F\cap{E}}|f|d\mu+\int_{F\cap E_1}|T_1(f\chi_{E_1})|d\mu.
	\end{align*}
	Next, by assumption that $T_j$ is a substochastic operator on $L^1+L^\infty$ and by Lemma 2.5 in\cite{BS} we have
	\begin{align*}
		\int_{F\cap E_1}|T_1(f\chi_{E_1})|d\mu&\leq \mu(F\cap E_1) (T_1(f\chi_{E_1}))^{**}(\mu(F\cap E_1))d\mu \\
		&\leq\mu(F\cap E_1)(f\chi_{E_1})^{**}(\mu(F\cap E_1))d\mu\\
		&=\int_0^{\mu(F\cap E_1)}(f\chi_{E_1})^*(s)ds.
	\end{align*}
	Since $t_0=\mu(F\cap E_1)\leq\mu(E_1)$, by Proposition 3.3(b) in \cite{BS} there exists a subset $B\subset{E_1}$ such that $\mu(B)=t_0$ and 
	\begin{equation*}
		\int_0^{\mu(F\cap E_1)}(f\chi_{E_1})^*(s)ds=\int_{B}|f\chi_{E_1}|d\mu=\int_{B}|f|d\mu.
	\end{equation*}
	In consequence, we get
	\begin{align*}
		\int_{F}|\mathcal{A}f|d\mu&\leq \int_{F\cap{E}}|f|d\mu+\int_0^{\mu(F\cap E_1)}(f\chi_{E_1})^*(s)ds\\
		&\leq\int_{F\cap{E}}|f|d\mu+\int_{B}|f|d\mu=\int_{B\cup F\cap{E}}|f|d\mu
	\end{align*}
	Next, since $B\cap (F\cap E)=\emptyset$ and	$\mu(F\cap E \cup B)=\mu(E\cup F)+\mu(B)=t$, by the Hardy-Littlewood inequality (Theorem 2.2 in \cite{BS}) it follows that 
	\begin{equation*}
		\int_F|\mathcal{A}f|d\mu\leq \int_{0}^t f^*(s)ds.
	\end{equation*}
	In consequence, by Proposition 3.3 \cite{BS} taking supremum over all sets $F$ of a measure $t$ we conclude that for all $t>0$,
	\begin{equation*}
		\int_{0}^t(\mathcal{A}f)^*(s)ds\leq\int_0^t f^*(s)ds.
	\end{equation*}
	Now, we consider the case when $J=\{1,2,3,\dots,N\}$ is a finite index set. Define 
	\begin{equation*}
		I_n=I\setminus\bigcup_{m=1}^nE_m\quad\textnormal{and}\quad \mathcal{A}_nf=f\chi_E+\sum_{m=1}^{n}T_m(f\chi_{E_m})\chi_{E_m}
	\end{equation*}
	for all $n\in{J}$. We easily observe that 
	\begin{equation*}
		\mathcal{A}_nf\prec\mathcal{A}_{n-1}f\prec\mathcal{A}_{n-2}f\prec\dots\prec\mathcal{A}_1f\prec f.
	\end{equation*}
	Now, assume that $J=\{1,2,\dots\}$ is countable infinite. We claim that $|\mathcal{A}f|\leq\mathcal{A}|f|$. Indeed, by assumption that $T_j$ is a positive contraction on a pair $(L^1,L^\infty)$ and since $-|f|\leq{f}\leq|f|$ it follows that $-\mathcal{A}|f|\leq{\mathcal{A}f}\leq\mathcal{A}|f|$. Now, it is sufficient to prove that $\mathcal{A}f\prec{f}$ for all nonnegative $f$. Consider for all $n\in\mathbb{N}$,
	\begin{equation*}
		f_n=f\chi_{E}+\sum_{m=1}^{n}f\chi_{E_m}.
	\end{equation*}
	It is easy to see that $0\leq{f_n}\uparrow{f}$ a.e. and since $T_j$ for each $j\in{J}$ is a positive contraction we provide
	\begin{align*}
		\mathcal{A}f_n&=\mathcal{A}(f\chi_E)+\mathcal{A}\left(\sum_{m=1}^{n}f\chi_{E_m}\right)\\
		&=f\chi_E+\sum_{m=1}^{n}T_m\left(f\chi_{E_m}\right)\chi_{E_m}\rightarrow f\chi_E+\sum_{m\in J}T_m\left(f\chi_{E_m}\right)\chi_{E_m}
	\end{align*}
	a.e. on $I$. Moreover, we have 
	\begin{equation*}
		\mathcal{A}f_n=\mathcal{A}_nf_n=\mathcal{A}_nf\prec{f} 
	\end{equation*}
	for all $n\in\mathbb{N}$, whence by Proposition 3.2 \cite{BS} this yields $\mathcal{A}f\prec{f}$.
\end{proof}

\begin{remark}
	Let us observe that the above theorem is a more general version of the well-known  Proposition 3.7 in \cite{BS} shown for integer operators. Namely, in Proposition \ref{prop:sum:substoch:operat} replacing the operator $T_j$ by the integral operator
	\begin{equation*}
		\frac{1}{\mu(E_j)}\int_{E_j}{x(t)}d\mu(t),
	\end{equation*}
	given for all integrable $x$ on each $E_j$, we constitute the conclusion immediately.
\end{remark}

\begin{remark}
	Now, we recall some examples of substochastic operators. First, we observe that for any measure preserving transformation $\sigma:I\rightarrow{I}$ there exists a substochastic operator $T_\sigma:L^1+L^\infty\rightarrow{L^1+L^\infty}$ given by $T_\sigma(x)=x\circ\sigma$ for any $x\in{L^1+L^\infty}$. Namely, $T_\sigma(x)\prec x$ for any $x\in{L^1+L^\infty}$. Next, taking $T=\{e^{i\theta}:-\pi\leq\theta\leq\pi\}$, by Lemma 6.1 in \cite{BS}, we notice that the convolution given for any $f\in{L^1(T)}$ by
	$$(f*g)(e^{i\theta})=\int_{-\pi}^{\pi}f(e^{i(\theta-\phi)})g(e^{i\phi})d\phi$$
	where $g\in L^1(T)$ with $\norm{g}{L^1}{}=1$ is a substochastic operator in $L^1(T)$ i.e. for any $f\in{L^1(T)}$ we have $(f*g)\prec\norm{g}{L^1}{}f=f$. Furthermore, by Theorem 3.8 in \cite{BS} we see that the Hardy-Littlewood maximal operator defined for any locally integrable function $f$ on $\mathbb{R}^n$ by
	$$Mf(t)=\sup_{\Omega\in t}\frac{1}{\mu(\Omega)}\int_{\Omega}|f(s)|ds$$
	constitutes the substochastic operator $T(f)=cM(f)\prec{f}$ for any $f\in L^1(\mathbb{R}^n)+L^\infty(\mathbb{R}^n)$, where constant $c$ depends on $n\in\mathbb{N}$. 	
	Finally, ti is worth mentioning that any permutation $\gamma:{\mathbb{N}\rightarrow\mathbb{N}}$ generates a substochastic operator $P_\gamma(x)=x_{\gamma(n)}$ for all $x\in\ell^1+\ell^\infty$.
\end{remark}

Now, we discuss a relation between lower local uniform $K$-monotonicity and convergence of the sequence of substochastic operators in symmetric spaces.

\begin{theorem}\label{thm:approx:2}
	Let $E$ be a symmetric space on $I$. Assume that $(\mathcal{S}_n)$ is a sequence of operators given by
	\begin{equation*}
	\mathcal{S}_n(x)=x\chi_{\Omega_n}+\sum_{j\in J_n}\left(\frac{1}{\mu(E_j^{(n)})}\int_{E_j^{(n)}}xd\mu\right)\chi_{E_j^{(n)}}
	\end{equation*}
	for any $x\in{E}$, where for any $n\in\mathbb{N}$, $\mathcal{A}_{n}=\{E_j^{(n)}\}_{j\in J_n}$ is a countable collection of pairwise disjoint measurable subsets of positive and finite measure on $I$ such that each set of $\mathcal{A}_n$ is the union of a finite subcollection of $\mathcal{A}_{n+1}$ and
	\begin{equation*}
	\Omega_n=I\setminus\bigcup_{j\in J_n}E_j^{(n)},\qquad\lim_{n\rightarrow\infty}\sup_{j\in J_n}\mu(E_j^{(n)})=0.
	\end{equation*}
	If $E$ is lower locally uniformly $K$-monotone, then for any $x\in{E}$ we have $$\norm{\mathcal{S}_n(x)-x}{E}{}\rightarrow{0}\qquad\textnormal{as}\qquad{n\rightarrow\infty}.$$
	Additionally, there exists a sequence of finite rank operators $(T_n)$ from $E$ into $E$ such that 
	\begin{equation*}
	\norm{(\mathcal{S}_n-T_n)(x)}{}{}\rightarrow{0}\qquad\textnormal{as}\qquad{n\rightarrow\infty}
	\end{equation*}
	for any $x\in{E}$.
\end{theorem}
\begin{proof}
	First, we prove that for any $x\in{E}$,
	\begin{equation}\label{equ:1:thm:approx:2}
	\norm{\mathcal{S}_n(x)^*-x^*}{E}{}\rightarrow{0}\qquad\textnormal{as}\qquad{n\rightarrow\infty}.
	\end{equation}
	Let $|J_n|$ means the cardinality of the set $J_n$ for any $n\in\mathbb{N}$. Then, by definition of a sequence $(\mathcal{A}_n)$ we are able to find a sequence $(\sigma_n)\subset\mathbb{N}$ such that $\sigma_n\leq|J_n|$ and for every $n\in\mathbb{N}$ we have
	\begin{equation*}
	\bigcup_{j=1}^{\sigma_n}E_j^{(n)}\subset\bigcup_{j=1}^{\sigma_{n+1}}E_j^{(n+1)}.
	\end{equation*}
	We define the finite rank operator $T_n:E\rightarrow{E}$ that corresponds to $\mathcal{S}_n$ given by
	\begin{equation*}
	T_n(x)=x\chi_{\Theta_n}+\sum_{j=1}^{\sigma_n}\left(\frac{1}{\mu(E_j^{(n)})}\int_{E_j^{(n)}}xd\mu\right)\chi_{E_j^{(n)}}
	\end{equation*}
	for any $x\in{E}$, where  $\Theta_n=\Omega_n\cap[0,n]$. It is easy to see that $\abs{T_n(x)}{}{}\leq\abs{\mathcal{S}_n(x)}{}{}$ for any $n\in\mathbb{N}$ and $x\in{E}$. Furthermore, by Proposition 3.7 \cite{BS} it follows that 
	\begin{equation}\label{equ:2:thm:approx:2}
	\mathcal{S}_n(x)\prec{x}\qquad\textnormal{for all }n\in\mathbb{N}\textnormal{ and }x\in{E}.
	\end{equation}
	Hence, since $E$ is symmetric we conclude
	\begin{equation}\label{inequ:thm:approx:2}
	\norm{T_n(x)}{}{}\leq\norm{\mathcal{S}_n(x)}{E}{}\leq\norm{x}{E}{}
	\end{equation}
	for any $n\in\mathbb{N}$ and $x\in{E}$. Next, since $x\in{E}$ is locally integrable on $I$, in view of Lebesgue's differentiation theorem (see Theorem 3.4 \cite{BS}) and by assumption that
	\begin{equation}\label{measure:thm:approx:2}
	\lim_{n\rightarrow\infty}\sup_{j\in J_n}\mu(E_j^{(n)})=0
	\end{equation}
	we get that for a.e. on $I$,
	\begin{equation}\label{conver:thm:approx:2}
	\mathcal{S}_n(x)\rightarrow{x}\qquad\textnormal{and}\qquad{T}_n(x)\rightarrow{x}\qquad\textnormal{as}\qquad n\rightarrow\infty.
	\end{equation} 
	Moreover, since $E$ is lower locally uniformly $K$-monotone, by Proposition 3.4 \cite{Cies-LLUKM} this yields that $E$ is order continuous.	Now, for any fixed $x\in{E}$ we can find a sequence of bounded functions $(x_n)\subset{E}$ that corresponds to $(T_n)$ such that $|x_n|\leq\min\{|x|,{n}\}$ and $T_n(x_n)=x_n$ for every $n\in\mathbb{N}$ and $x_n\rightarrow{x}$ a.e. on $I$ as $n\rightarrow\infty$. Namely, taking $(y_n)$ a sequence of truncations of the element $x$ we may construct $(x_n)$ by
	\begin{equation*}
	x_n=y_n\chi_{\Theta_n}+\sum_{j=1}^{\sigma_n}c_j^{(n)}\chi_{E_j^{(n)}}
	\end{equation*}	  
	where a sequence $(c_j^{(n)})\subset\mathbb{R}$ is chosen in such way that $|x_n|\leq|y_n|\leq|x|$ on $I$ and it guarantees that $x_n$ converges to $x$ a.e. on $I.$ Additionally, by definition of $x_n$ on the set $\bigcup_{j=1}^{\sigma_n}E_j^{(n)}$ it is easy to observe that $T_n(x_n)=x_n$ and $|x_n|\leq{|T_n(x)|}$ a.e. on $I$ for all $n\in\mathbb{N}$. Therefore, by order continuity and symmetry of $E$ we have 
	\begin{equation*}
		\norm{x_n}{E}{}\leq\norm{T_n(x)}{E}{}\quad\textnormal{for all }n\in\mathbb{N},\qquad\norm{x-x_n}{E}{}\rightarrow{0}\quad\textnormal{as}\quad n\rightarrow\infty.
	\end{equation*}
	Hence, by \eqref{inequ:thm:approx:2} we conclude 
	\begin{equation*}
	\lim_{n\rightarrow\infty}\norm{T_n(x)}{E}{}=\lim_{n\rightarrow\infty}\norm{\mathcal{S}_n(x)}{E}{}=\norm{x}{E}{}.
	\end{equation*}
	Thus, since $E$ is lower locally uniformly $K$-monotone, in view of \eqref{equ:2:thm:approx:2} we show \eqref{equ:1:thm:approx:2}. Now, we prove that $\mathcal{S}_n(x)$ converges to $x$ globally in measure on $I$. Since $E$ is order continuous, by Lemma 2.5 \cite{CieKolPan} it follows that $\lim_{t\rightarrow{a}}x^*(t)=x^*(a)=0$. So, for any $\epsilon>0$ there exists a subset $\Omega\subset{I}$ such that $|x(t)|<\epsilon$ for any $t\in\Omega$. It is easy to notice that $\mu(\Omega^c)=d_{x}(\epsilon)<\infty$ and $\mathcal{S}_n(x)=x$ on $\Omega_n$ for all $n\in\mathbb{N}$, whence
	\begin{align*}
	\mu\left\{s\in{I}:\abs{\mathcal{S}_n(x)(s)-x(s)}{}{}\geq\epsilon\right\}
	=&\mu\left\{s\in{\Omega^c}\cap\Omega_n^c:\abs{\mathcal{S}_n(x)(s)-x(s)}{}{}\geq\epsilon\right\}\\
	&+\mu\left\{s\in{\Omega}\cap\Omega_n^c:\abs{\mathcal{S}_n(x)(s)-x(s)}{}{}\geq\epsilon\right\}
	\end{align*} 
	for every $n\in\mathbb{N}$. Next, in view of \eqref{conver:thm:approx:2} we observe that 
	\begin{equation}\label{equ:4:approx:2}
	\mu\left\{s\in{\Omega^c}\cap\Omega_n^c:\abs{\mathcal{S}_n(x)(s)-x(s)}{}{}\geq\epsilon\right\}\rightarrow{0}\quad\textnormal{as}\quad{n\rightarrow\infty}.
	\end{equation} 
	We suppose temporarily for simplicity of our investigation that $x$ is a simple function with finite measure support. Then, without loss of generality we may consider that 
	\begin{equation*}
	x=\sum_{j=1}^ma_i\chi_{B_j}
	\end{equation*}
	where $\{B_1,\dots,B_m\}\subset{I}$ are pairwise disjoint subsets of finite positive measure and $a_j\in\mathbb{R}$ for any $j\in\{1,\dots,m\}$. Thus, by Proposition 8 and Proposition 15 in \cite{Royd} for any $j\in\{1,\dots,m\}$ there exists a finite collection of open intervals $\{U_i\}_{i=1}^{m_j}\subset{I}$ such that 
	\begin{equation*}
	B_j\subset\bigcup_{i=1}^{m_j}U_i\qquad\textnormal{and}\qquad\mu\left(\bigcup_{i=1}^{m_j}U_i\setminus{B_j}\right)<\frac{\epsilon}{m}
	\end{equation*}
	for any $j\in\{1,\dots,m\}$. Thus, there exists $0<\gamma<\infty$ such that $\supp(x)\subset[0,\gamma]$. Hence, by \eqref{conver:thm:approx:2} we get $\mathcal{S}_n(x)$ converges to $x$ in measure on $[0,2\gamma]$, which implies that 
	\begin{equation*}
	\mu\left\{s\in{\Omega}\cap\Omega_n^c:\abs{\mathcal{S}_n(x)(s)-x(s)}{}{}\geq\epsilon\right\}\rightarrow{0}\qquad\textnormal{as}\qquad{n\rightarrow\infty.}
	\end{equation*}
	In consequence, by \eqref{equ:4:approx:2} we obtain
	\begin{equation*}
	\mu\left\{s\in{I}:\abs{\mathcal{S}_n(x)(s)-x(s)}{}{}\geq\epsilon\right\}\rightarrow{0}\qquad\textnormal{as}\qquad{n\rightarrow\infty,}
	\end{equation*}
	which proves that $\mathcal{S}_n(x)$ converges to $x$ globally in measure on $I$ in case when $x$ is a simple function with finite measure support. Therefore, since $E$ is order continuous, by \eqref{equ:1:thm:approx:2} and by Proposition 2.4 in \cite{CzeKam} it follows that
	\begin{equation}\label{equ:5:approx:2}
	\norm{\mathcal{S}_n(x)-x}{E}{}\rightarrow{0}\qquad\textnormal{as}\qquad{n\rightarrow\infty}
	\end{equation}
	for any simple function $x$ with finite measure support. Finally, taking arbitrary $x$ in $E$ we may find $(y_n)$ a sequence of simple functions with finite measure support  such that $0\leq|y_n|\leq|x|$ a.e. on $I$ for any $n\in\mathbb{N}$ and $y_n\rightarrow{x}$ a.e. on $I$ as $n\rightarrow\infty$. So, by order continuity of $E$ there is $M_\epsilon\in\mathbb{N}$ such that for all $m\geq{M_\epsilon}$ we have $\norm{x-y_m}{E}{}<\epsilon/3$. Next, by \eqref{inequ:thm:approx:2} and by definition of $\mathcal{S}_n$ it is easy to see that 
	\begin{equation*}
	\norm{\mathcal{S}_n(x)-\mathcal{S}_n(y_m)}{E}{}=\norm{\mathcal{S}_n(x-y_m)}{E}{}\leq\norm{x-y_m}{E}{}
	\end{equation*} 
	for any $n,m\in\mathbb{N}$. Therefore, choosing $m\geq{M_\epsilon}$ we observe that
	\begin{align*}
	\norm{\mathcal{S}_n(x)-x}{E}{}&\leq\norm{\mathcal{S}_n(x)-\mathcal{S}_n(y_m)}{E}{}+\norm{\mathcal{S}_n(y_m)-y_m}{E}{}+\norm{y_m-x}{E}{}\\
	&=\norm{\mathcal{S}_n(x-y_m)}{E}{}+\norm{\mathcal{S}_n(y_m)-y_m}{E}{}+\norm{y_m-x}{E}{}\\
	&\leq{2}\norm{x-y_m}{E}{}+\norm{\mathcal{S}_n(y_m)-y_m}{E}{}\\
	&\leq\frac{2\epsilon}{3}+\norm{\mathcal{S}_n(y_m)-y_m}{E}{}.
	\end{align*}
	Hence, by \eqref{equ:5:approx:2} there exists $N_\epsilon\in\mathbb{N}$ such that for any $n\geq{N_\epsilon}$ we have
	\begin{align*}
	\norm{\mathcal{S}_n(x)-x}{E}{}\leq\frac{2\epsilon}{3}+\norm{\mathcal{S}_n(y_m)-y_m}{E}{}<\epsilon,
	\end{align*}
	which shows that $\mathcal{S}_n(x)$ converges to $x$ in $E$. Now, proceeding analogously for a sequence $(T_n)$ as in the case of a sequence $(\mathcal{S}_n)$ we provide that $T_n$ converges to $x$ in $E$, which yields that for any $x\in{E}$, $$\norm{(T_n-\mathcal{S}_n)(x)}{}{}\rightarrow{0}\qquad\textnormal{as}\qquad{n\rightarrow\infty}.$$
\end{proof}

In the next theorem we investigate criteria which guarantee that the generalization of Lebesgue's differentiation theorem holds in symmetric spaces. 

\begin{theorem}\label{thm:approx:3}
	Let $E$ be a symmetric space on $I$. Assume that $(\mathcal{H}_n)$ is a sequence of operators given by
	\begin{equation*}
	\mathcal{H}_n(x)=\sum_{j\in J_n}\left(\frac{1}{\mu(E_j^{(n)})}\int_{E_j^{(n)}}xd\mu\right)\chi_{E_j^{(n)}}
	\end{equation*}
	for any $x\in{E}$, where for any $n\in\mathbb{N}$, $\mathcal{A}_{n}=\{E_j^{(n)}\}_{j\in J_n}$ is a countable collection of pairwise disjoint measurable subsets of positive and finite measure on $I$ such that each set of $\mathcal{A}_n$ is the union of a finite subcollection of $\mathcal{A}_{n+1}$ and
	\begin{equation*}
	\Omega_n=I\setminus\bigcup_{j\in J_n}E_j^{(n)},\qquad\lim_{n\rightarrow\infty}\mu\left(\Omega_n\right)=0,\qquad \lim_{n\rightarrow\infty}\sup_{j\in J_n}\mu(E_j^{(n)})=0.
	\end{equation*}
	If $E$ is lower locally uniformly $K$-monotone, then for any $x\in{E}$ we have $$\norm{\mathcal{H}_n(x)-x}{E}{}\rightarrow{0}\qquad\textnormal{as}\qquad{n\rightarrow\infty}.$$
	Additionally, there exists a sequence of finite rank operators $(T_n)$ from $E$ into $E$ such that 
	\begin{equation*}
	\norm{(\mathcal{H}_n-T_n)(x)}{}{}\rightarrow{0}\qquad\textnormal{as}\qquad{n\rightarrow\infty}.
	\end{equation*}
\end{theorem}

\begin{proof}
	Immediately, since $E$ is lower locally uniformly $K$-monotone, applying Proposition 3.4 \cite{Cies-LLUKM} we obtain $E$ is order continuous. Therefore, since $\mu(\Omega_n)\rightarrow{0}$ as $n\rightarrow\infty$, we conclude that 
	\begin{equation*}
	\norm{x\chi_{\Omega_n}}{E}{}\rightarrow{0}\qquad\textnormal{as }\qquad n\rightarrow\infty. 
	\end{equation*}
	Next, by assumption that $\Omega_n=I\setminus\bigcup_{j\in J_n}E_j^{(n)}$ and by definition of $\mathcal{H}_n$ for any $n\in\mathbb{N}$ it follows that 
	\begin{equation*}
		\norm{\mathcal{H}_n(x)-x}{E}{}=\norm{(\mathcal{H}_n(x)-x)\chi_{\Omega_n^c}+x\chi_{\Omega_n}}{E}{}\leq\norm{\mathcal{H}_n(x)-x\chi_{\Omega_n^c}}{E}{}+\norm{x\chi_{\Omega_n}}{E}{}
	\end{equation*}
	for all $n\in\mathbb{N}$. Hence, by Theorem \ref{thm:approx:2} we have
	\begin{equation*}
		\norm{\mathcal{H}_n(x)-x}{E}{}\rightarrow{0}\qquad\textnormal{as}\qquad{n\rightarrow\infty}.
	\end{equation*}
	We can easily observe that there exists a sequence of finite rank operators $(T_n)$, that is given analogously as in the previous theorem and corresponds to a sequence $(\mathcal{H}_n)$, such that for any $n\in\mathbb{N}$ we 
	\begin{equation*} 
		\norm{\mathcal{H}_n(x)-T_n(x)}{E}{}\leq	\norm{\mathcal{H}_n(x)+x\chi_{\Omega_n}-T_n(x)}{E}{}+\norm{x\chi_{\Omega_n}}{E}{}.
	\end{equation*} 
	In consequence, by definition of $\mathcal{H}_n$ for any $n\in\mathbb{N}$ and by Theorem \ref{thm:approx:2} it follows that $\norm{\mathcal{H}_n(x)-T_n(x)}{E}{}\rightarrow{0}$ as $n\rightarrow\infty$ and ends the proof.
\end{proof}


In the next theorems we investigate a convergence of admissible operators in symmetric spaces. 

\begin{theorem}\label{thm:1:substoch:DUKM}
	Let $E$ be a symmetric space and $(T_n)$ be a sequence of admissible operators on $L^1+L^\infty$ such that $T_{n+1}\prec{T_n}$ for every $n\in\mathbb{N}$. If $E$ is decreasing uniformly $K$-monotone then for any $x\in{E}$ there exists $y\in{E}$ such that 
	\begin{equation*}
		\norm{(T_nx)^*-y^*}{E}{}\rightarrow{0}\quad\textnormal{ and }\quad (T_nx)^*\rightarrow{y^*}\quad\textnormal{globally in measure as }n\rightarrow\infty.
	\end{equation*}
\end{theorem}

\begin{proof}
	Let $x\in{E}$. Since $T_n$ is admissible and $T_{n+1}\prec{T_n}$ for every $n\in\mathbb{N}$ there exists a constant $C$ such that 
	\begin{equation*}
		T_{n+1}x\prec{T_nx}\prec{T_1x}\prec{Cx}.
	\end{equation*}
	for all $n\in\mathbb{N}$. Hence, by symmetry of $E$ we get for any $n\in\mathbb{N}$,
	\begin{equation*}
		\norm{T_{n+1}x}{E}{}\leq\norm{T_nx}{E}{}\leq\norm{Cx}{E}{}.
	\end{equation*}
	Thus, since any monotone and bounded real sequence is convergent we obtain that 
	\begin{equation}\label{equ:1:thm:Substoch:DUKM}
		\lim_{n\rightarrow\infty}\norm{T_nx}{E}{}=d<\infty.
	\end{equation}
	Next, taking two arbitrary sequences $(k_n),(j_n)\subset\mathbb{N}$ such that $k_n\leq j_n$ for each $n\in\mathbb{N}$ and denoting $x_n=T_{k_n}x$ and  $y_n=T_{j_n}x$ for any $n\in\mathbb{N}$ we can easily see that $x_{n+1}\prec x_n\prec{y_n}$ for any $n\in\mathbb{N}$. Therefore, by assumption that $E$ is decreasing uniformly $K$-monotone and by the fact \eqref{equ:1:thm:Substoch:DUKM} it follows that
	\begin{equation*}
		\norm{(T_{j_n}x)^*-(T_{k_n}x)^*}{E}{}=\norm{y_n^*-x_n^*}{E}{}\rightarrow{0}\quad\textnormal{as}\quad n\rightarrow\infty
	\end{equation*}
	So, since $(T_nx)^*$ is Cauchy in $E$, by completeness of $E$ there exists $y\in{E}$ such that 
 	\begin{equation*}
 	\norm{(T_nx)^*-y}{E}{}\rightarrow{0}\quad\textnormal{as}\quad n\rightarrow\infty.
 	\end{equation*}
 	Next, by Lorentz and Shimogaki Theorem (see Theorem 7.4 in \cite{BS}) we conclude that $y=y^*$ a.e. on $I$. Finally, by Proposition 4.11 in \cite{CieLewUKM}, since $E$ is a symmetric Banach function space, applying double extract sequence theorem we get that
 	\begin{equation*}
 		{(T_nx)^*}\rightarrow{y^*}\quad\textnormal{ globally in measure as}\quad n\rightarrow\infty.
 	\end{equation*}
\end{proof}

\begin{theorem}\label{thm:2:substoch:DUKM}
	Let $E$ be a symmetric space and $A$ be a substochastic operator on $L^1+L^\infty$. If $E$ is decreasing uniformly $K$-monotone then for any $x\in{E^+}$ there exists a sequence of substochastic operators $(B_n)_{n\geq 0}$ such that for every $n\in\mathbb{N}$,
	\begin{equation*}
		B_n\circ{A^n}x=B_0x \textnormal{ a.e.} \quad \textnormal{and}\quad\norm{(B_0x)^*-(A^nx)^*}{E}{}\rightarrow{0}\quad\textnormal{ as }\quad n\rightarrow\infty.
	\end{equation*}
	Additionally, if $\norm{A^nx}{E}{}\rightarrow\norm{x}{E}{}$ then 
	\begin{equation*}
			B_n\circ{A^n}x=x \textnormal{ a.e.}\quad \textnormal{and}\quad	\norm{({A^n}x)^*-x^*}{E}{}\rightarrow{0}\quad\textnormal{ as }\quad n\rightarrow\infty.
	\end{equation*}
\end{theorem}

\begin{proof}
	First, by Proposition 4.11 in \cite{CieLewUKM} and by Theorem \ref{thm:1:substoch:DUKM} we conclude that 
	\begin{equation}\label{equ:1:thm:2:substoch}
		\norm{(A^nx)^*-y^*}{E}{}\rightarrow{0}\quad\textnormal{ and }\quad (A^nx)^*\rightarrow{y^*}\quad\textnormal{a.e. as }n\rightarrow\infty.
	\end{equation} 
	We claim that $y\prec{x}$. Indeed, taking an arbitrary $t\in I$ and assuming that $\phi_E$ is the fundamental function of the space $E$, by Proposition 5.9 in \cite{BS} and by subadditivity of the Hardy operator (see Theorem 3.4 in \cite{BS}) we get
	\begin{equation*}
		|(A^nx)^{**}(t)-y^{**}(t)|\phi_E(t)\leq((A^nx)^*-y^*)^{**}(t)\phi_E(t)\leq\norm{(A^nx)^*-y^*}{E}{}.
	\end{equation*}
	In consequence, $(A^nx)^{**}(t)\rightarrow{y^{**}(t)}$ as $n\rightarrow\infty$ for any $t\in{I}$. Hence, since $A^nx\prec{x}$ for every $n\in\mathbb{N}$ we provide the claim. Thus, since $x\in{E}^+$, by Calder\'{o}n and Ryff Theorem (see Theorem 2.10 in \cite{BS}) there exists a substochastic operator $B_0$ on $L^1+L^\infty$ such that $B_0x=|y|$ a.e. Furthermore, since $A^{n+1}x\prec{A^nx}$ for any $n\in\mathbb{N}$, it follows that
	\begin{equation*}
		(B_0x)^{**}(t)=y^{**}(t)=\lim_{n\rightarrow\infty}(A^nx)^{**}(t)\leq(A^nx)^{**}(t)
	\end{equation*}
	for any $t\in{I}$ and $n\in\mathbb{N}$. Applying again Calder\'{o}n and Ryff Theorem we employ a sequence of the substochastic operators $(B_n)_{n\geq 1}$ such that  
	\begin{equation*}
		B_n\circ A^nx=B_0x=|y|
	\end{equation*}
a.e. on $I$ and	for every $n\in\mathbb{N}$. Hence, applying \eqref{equ:1:thm:2:substoch} we prove the first part of our theorem. Now, assume that $\norm{A^nx}{E}{}\rightarrow\norm{x}{E}{}$. Then, since $y\prec{x}$, by assumption that $E$ is decreasing uniformly $K$-monotone and by Remark 4.1 in \cite{CieLewUKM} it follows that $x=|y|$ a.e. on $I$ and completes the proof.
\end{proof}

The immediate consequence of the previous theorem is the following result.

\begin{corollary}
	Let $E$ be a symmetric space with the fundamental function $\phi$ such that $\phi(\infty)=\infty$ and $A$ be a substochastic operator on $L^1+L^\infty$. If $E$ is decreasing uniformly $K$-monotone and if for any $x\in{E^+}$ we have $Ax\leq{x}$ and $\norm{A^nx}{E}{}\rightarrow\norm{x}{E}{}$ as $n\rightarrow\infty$ then $\norm{A^nx-x}{E}{}\rightarrow{0}$ as $n\rightarrow\infty$.  
\end{corollary}

\begin{proof}
	Directly, by Theorem \ref{thm:2:substoch:DUKM} we obtain 
	\begin{equation}\label{equ:1:coro:Substoch}
	\norm{x^*-(A^nx)^*}{E}{}\rightarrow{0}\quad\textnormal{ as }\quad n\rightarrow\infty.
	\end{equation}
	Next, since $Ax\leq{x}$ for any $x\in{E^+}$ it follows that
	\begin{equation*}
		A^nx\leq\frac{A^nx+x}{2}\leq{x}
	\end{equation*}
	for every $n\in\mathbb{N}$. Hence, by symmetry of $E$ and by assumption that $\norm{A^nx}{E}{}\rightarrow\norm{x}{E}{}$ as $n\rightarrow\infty$ we have 
	\begin{equation*}
		\frac{1}{2}\norm{{A^n}x+x}{E}{}\rightarrow\norm{x}{E}{}\quad\textnormal{ as }\quad n\rightarrow\infty.
	\end{equation*}
	In consequence, since $\frac{1}{2}(A^nx+x)\prec{x}$ for all $n\in\mathbb{N}$, in view of the assumption that $E$ is decreasing uniformly $K$-monotone we conclude that 
	\begin{equation*}
		\norm{\frac{1}{2}({A^n}x+x)^*-x^*}{E}{}\rightarrow{0}\quad\textnormal{ as }\quad n\rightarrow\infty.
	\end{equation*}
	Therefore, by \eqref{equ:1:coro:Substoch} and by Lemma 2.2 in \cite{CzeKam} this yields that $A^nx\rightarrow{x}$ globally in measure. Furthermore, since $\phi(\infty)=\infty$, by Proposition 4.11 in \cite{CieLewUKM} we provide $E$ is order continuous. Finally, by \eqref{equ:1:coro:Substoch} and by Proposition 2.4 in \cite{CzeKam} we get end of the proof.
\end{proof}

Now we focus our attention on conditions that guarantee the convergence of substochastic operators in symmetric spaces.

\begin{theorem}\label{thm:conver:Sub-operstor}
	Let $X,Y$ be symmetric spaces and let $T_n,T\in B(X,Y)$ be substochastic operators such that $T_n\leq{T_{n+1}}\leq{T}$ for all $n\in\mathbb{N}$. If $Y$ is uniformly $K$-monotone, then 
	\begin{equation*}
	\norm{T_n(x)-T(x)}{Y}{}\rightarrow{0}\quad\textnormal{as}\quad n\rightarrow\infty,
	\end{equation*}
	for some $x\in{X^+}$ such that $\norm{x}{X}{}=\norm{T(x)}{Y}{}=1$. 
\end{theorem}

\begin{proof}
	Let $\epsilon>0$. Then, there exists $x_1\in X$ such that $\norm{x_1}{X}{}=1$, $x_1\geq{0}$ and 
	\begin{equation*}
	\abs{\norm{T_1(x_1)}{Y}{}-\norm{T_1}{}{}}{}{}<\frac{\epsilon}{2}.
	\end{equation*}
	Moreover, without loss of generality passing to subsequence and relabeling we may assume that for all $n\in\mathbb{N}$, 
	\begin{equation*}
	\abs{\norm{T_n}{}{}-1}{}{}<\frac{\epsilon}{2},
	\end{equation*}
	whence we get
	\begin{equation}\label{equ1:thm:CSO}
	\abs{\norm{T_1(x_1)}{Y}{}-1}{}{}<{\epsilon}.
	\end{equation}
	Next, since $T_n\leq{T_{n+1}}\leq{T}$ for all $n\in\mathbb{N}$, it follows that 
	\begin{equation*}
		 T_1(x_1)\leq{T_n(x_1)}\leq{T_{n+1}(x_1)}\leq{T(x_1)},
	\end{equation*}
	for any $n\in\mathbb{N}$. By symmetry of $Y$, we conclude 
	\begin{equation}\label{equ2:thm:CSO}
		 \norm{T_1(x_1)}{Y}{}\leq\norm{T_n(x_1)}{Y}{}\leq\norm{T_{n+1}(x_1)}{Y}{}\leq\norm{T(x_1)}{Y}{},
	\end{equation}
	for any $n\in\mathbb{N}$. Hence, in view of \eqref{equ1:thm:CSO} we easily get 
	\begin{equation*}
		1-\epsilon<\norm{T_1(x_1)}{Y}{}\leq\norm{T_n(x_1)}{Y}{}\leq\norm{T(x_1)}{Y}{}\leq{1},
	\end{equation*}
	for all $n\in\mathbb{N}$. So, since $\epsilon$ is arbitrary chosen we obtain
	\begin{equation}\label{equ:3:thm:CSO}
		\norm{T_n(x_1)}{Y}{}\rightarrow\norm{T(x_1)}{Y}{}=1\quad\textnormal{as}\quad n\rightarrow\infty.
	\end{equation}
	Now, applying analogous techniques as in the proof of Theorem \ref{thm:full:descr:IUKM} and passing to subsequence and also relabelling if necessary we may show that there exists $S(x_1)\in{Y}$ such that 
	\begin{equation}\label{equ:4:thm:CSO}
		{T_n(x_1)}^{*}\rightarrow{S(x_1)}^{*}\quad \textnormal{a.e. on $I$,}\qquad {T_n(x_1)}^{**}(t)\rightarrow{S(x_1)}^{**}(t) \quad\textnormal{for all }t\in{I}
	\end{equation}
	as $n\rightarrow\infty$ and 
	\begin{equation}\label{equ:5:thm:CSO}
		\norm{T_n(x_1)}{Y}{}\rightarrow\norm{S(x_1)}{Y}{}\quad\textnormal{as}\quad n\rightarrow\infty \qquad\textnormal{and}\qquad T_n(x_1)\prec{S(x_1)}
	\end{equation}
	for each $n\in\mathbb{N}$. Furthermore, since $T_n\leq{T}$ for all $n\in\mathbb{N}$ on the positive cone of $X$ we have $T_n(x_1)^*\leq{T(x_1)}^*$. Hence, by \eqref{equ:3:thm:CSO}, \eqref{equ:4:thm:CSO} and \eqref{equ:5:thm:CSO} this yields $S(x_1)^*\leq{T(x_1)^*}$ a.e. on $I$ and $\norm{S(x_1)}{y}{}=1$. Consequently, since $Y$ is uniformly $K$-monotone by Remark 4.1 in \cite{CieLewUKM} it follows that $Y$ is strictly $K$-monotone and so $T(x_1)^*=S(x_1)^*$ a.e. on $I$. Moreover, it is clearly obvious that $T_n(x_1)\prec T_{n+1}(x_1)\prec T(x_1)$ for any $n\in\mathbb{N}$ and in view of \eqref{equ:3:thm:CSO} and by assumption that $Y$ is uniformly $K$-monotone we get 
	\begin{equation}\label{equ:6:thm:CSO}
		\norm{T_n(x_1)^*-T(x_1)^*}{Y}{}\rightarrow{0}\quad\textnormal{as}\quad n\rightarrow\infty.
	\end{equation}
	Now, we shall prove that 
	\begin{equation}\label{equ:7:thm:CSO}
	\norm{\frac{(T_n(x_1)+T(x_1))}{2}^*-T(x_1)^*}{Y}{}\rightarrow{0}\quad\textnormal{as}\quad n\rightarrow\infty.
	\end{equation}
	First, we observe that for all $n\in\mathbb{N}$,
	\begin{equation*}
		T_n(x_1)\leq \frac{T_n(x_1)+T(x_1)}{2}\leq T(x_1).
	\end{equation*}
	Hence, by symmetry of $Y$ this implies 
	\begin{equation*}
		\norm{T_n(x_1)}{Y}{}\leq \frac{\norm{T_n(x_1)+T(x_1)}{Y}{}}{2}\leq \norm{T(x_1)}{Y}{}=1,
	\end{equation*}
	for all $n\in\mathbb{N}$, which provides in view of \eqref{equ:5:thm:CSO} that 
	\begin{equation*}
		\frac{\norm{T_n(x_1)+T(x_1)}{Y}{}}{2}\rightarrow{1}\quad\textnormal{as}\quad n\rightarrow\infty.
	\end{equation*}
	Moreover, for any $n\in\mathbb{N}$ we have $T_n(x_1)+T(x_1)\leq 2T(x_1)$. In consequence, proceeding analogously as in the proof of Theorem \ref{thm:full:descr:IUKM} we may see that
	\begin{equation*}
		\frac{T_n(x_1)+T(x_1)}{2}\prec T(x_1) \quad \textnormal{and}\quad \frac{(T_n(x_1)+T(x_1))}{2}^*\rightarrow T(x_1)^*\quad \textnormal{a.e. on }I.
	\end{equation*}
	Thus, by assumption that $Y$ is uniformly $K$-monotone we prove \eqref{equ:7:thm:CSO}. Next, according to Lemma 2.2 \cite{CzeKam} it follows that $T_n(x_1)$ converges to $T(x_1)$ globally in measure on $I$. Consequently, in view of Remark 4.1 in \cite{CieLewUKM} and by assumption that $Y$ is uniformly $K$-monotone we get $Y$ is order continuous. Finally, by \eqref{equ:6:thm:CSO} and by Proposition 2.4 \cite{CzeKam} we obtain that 
	\begin{equation*}
		\norm{T_n(x_1)-T(x_1)}{Y}{}\rightarrow{0}\quad\textnormal{as}\quad n\rightarrow\infty.
	\end{equation*}
\end{proof}

\begin{theorem}
	Let $X,Y$ be symmetric spaces and let $T_n,T\in{B(X,Y)}$ be substochastic operators such that for any $x\in{X^+}$ we have $\norm{T_n(x)}{Y}{}\rightarrow\norm{T(x)}{Y}{}$ and $T_n\leq{T_{n+1}}\leq{T}$ for all $n\in\mathbb{N}$. If $Y$ is uniformly $K$-monotone, then for any $x\in{X}$ we have 
	\begin{equation*}
	\norm{T_n(x)-T(x)}{Y}{}\rightarrow{0}\quad\textnormal{as}\quad n\rightarrow\infty.
	\end{equation*}	
\end{theorem}

\begin{proof}
	Since the convergence is satisfied for $x=0$ we may assume without loss of generality that $x\neq{0}$. 
	Proceeding analogously as in the proof of Theorem \ref{thm:conver:Sub-operstor} and assuming that $	\norm{T_n(x)}{Y}{}\rightarrow\norm{T(x)}{Y}{}=1$ as $n\rightarrow\infty$ we obtain that for any $x\in X^+\setminus\{0\}$,
	\begin{equation*}
		\norm{T_n(x)-T(x)}{Y}{}\rightarrow{0}.
	\end{equation*}
	Since for any substochastic operator $S:X\rightarrow Y$ we have $S(x)=S(x^+-x^-)=S(x^+)-S(x^-)$ where $x\in{X}$, we easily get for any $x\in{X}$,
	\begin{equation*}
		0\leq\norm{T_n(x)-T(x)}{Y}{}\leq	\norm{T_n(x^+)-T(x^+)}{Y}{}+\norm{T_n(x^-)-T(x^-)}{Y}{}\rightarrow{0}
	\end{equation*}
	as $n\rightarrow\infty$, which completes the proof.
\end{proof}

	\begin{theorem}\label{thm:admissible:compact}
		Let $\overline{X}=(X_0,X_1)$, $\overline{Y}=(Y_0,Y_1)$ be two Banach couples and let $X,Y$ be interpolation spaces with respect to the couples $\overline{X}$ and $\overline{Y}$, where $X,Y$ are of exponent $\theta\in[0,1]$. Assume that $Y_1$ is lower locally uniformly $K$-monotone and $T:\overline{X}\rightarrow\overline{Y}$ is an admissible operator such that $T:{X_1}\rightarrow{Y_1}$ compactly. Then $T:{X}\rightarrow{Y}$ is compact and there exists a sequence of admissible operators $T_n:\overline{X}\rightarrow\overline{Y}$ such that 
		\begin{equation*}
			\norm{T-T_n}{B(X,Y)}{}\rightarrow{0}\quad\textnormal{as}\quad{n}\rightarrow\infty.
		\end{equation*}
	\end{theorem}
	
	\begin{proof}
		Since every finite rank operator is compact and the norm-limit of compact operators is also a compact operator, it is enough to prove that $T:\overline{X}\rightarrow\overline{Y}$ an arbitrary admissible operator is a limit of the finite rank admissible operators. Define analogously as in Theorem \ref{thm:approx:2} the operator $F_n$ by 
		\begin{equation*}
			F_n(x)=x\chi_{\Omega_n}+\sum_{j\in J_n}\left(\frac{1}{\mu\left(E_j^{(n)}\right)}\int_{E_j^{(n)}}xd\mu\right)\chi_{E_j^{(n)}}
		\end{equation*}
		for any $n\in\mathbb{N}$ and $x\in{Y_0+Y_1}$. Let $T:{X_0+X_1}\rightarrow{Y_0+Y_1}$ be admissible operator and let $(X,Y)$ be an intermediate couple for $(X_0,X_1)$ and $(Y_0,Y_1)$. Denote $T_n=F_n\circ{T}$ for any $n\in\mathbb{N}$. Clearly, $T_n$ is a finite rank operator and
		\begin{equation}\label{equ:1:thm:ad:comp}
			\norm{T-T_n}{B(X_i,Y_i)}{}=\sup_{\norm{x}{X_i}{}\leq{1}}\norm{T(x)-F_n\circ{T}(x)}{Y_i}{}
		\end{equation}
		for any $n\in\mathbb{N}$. Since $V=\{T(z):\norm{z}{X_1}{}\leq{1}\}$ has a compact closure in $Y_1$, for each $v\in V$ there exists $(z_n)\in X$, $\norm{z_n}{X_1}{}\leq{1}$ such that $T(z_n)\rightarrow{v}$ as $n\rightarrow\infty$ in $Y_1$. Therefore, we have
		\begin{equation*}
			\sup_{\norm{x}{X_1}{}\leq{1}}\norm{T(x)-F_n\circ{T}(x)}{Y_1}{}=\sup_{v\in{V}}\norm{v-F_n(v)}{Y_1}{}
		\end{equation*}
		for any $n\in\mathbb{N}$. Next, by Theorem \ref{thm:approx:2}, since $Y_1$ is lower locally uniformly $K$-monotone, it follows that $\norm{v-F_n(v)}{Y_1}{}\rightarrow{0}$ as $n\rightarrow\infty$. Hence, since $V$ has a compact closure in $Y_1$, by \eqref{equ:1:thm:ad:comp} we get
		\begin{equation}\label{equ:2:thm:ad:comp}
			\norm{T-T_n}{B(X_1,Y_1)}{}=\sup_{v\in{V}}\norm{v-F_n(v)}{Y_1}{}\rightarrow{0}\quad\textnormal{as}\quad n\rightarrow\infty.
		\end{equation}		
		By assumption that $(X,Y)$ are the interpolation spaces of exponent $\theta\in[0,1]$, we conclude 
		\begin{align*}
			\norm{T-T_n}{B(X,Y)}{}&\leq\norm{T-T_n}{B(X_0,Y_0)}{1-\theta}\norm{T-T_n}{B(X_1,Y_1)}{\theta}\\
			&\leq\left(\norm{T}{}{}+\norm{T_n}{}{}\right)^{1-\theta}\norm{T-T_n}{B(X_1,Y_1)}{\theta}\\
			&\leq\left(\norm{T}{}{}+\norm{F_n}{}{}\norm{T}{}{}\right)^{1-\theta}\norm{T-T_n}{B(X_1,Y_1)}{\theta}.
		\end{align*} 
		Since $F_n(x_i)\prec{x_i}$ for all $n\in\mathbb{N}$ and $x_i\in{X_i}$, this yields that $\norm{F_n}{B(X_i,Y_i)}{}\leq{1}$, whence
		\begin{align*}
			\norm{T-T_n}{B(X,Y)}{}&\leq\left(2\norm{T}{B(X_0,Y_0)}{}\right)^{1-\theta}\norm{T-T_n}{B(X_1,Y_1)}{\theta}\\
			&\leq\left(2\norm{T}{B(X_0,Y_0)}{}\right)^{1-\theta}\left(2\norm{T}{B(X_1,Y_1)}{}\right)^{\theta}\\
			&=2\norm{T}{B(X_0,Y_0)}{1-\theta}\norm{T}{B(X_1,Y_1)}{\theta}.
		\end{align*}
		Moreover, by \eqref{equ:2:thm:ad:comp} we obtain
		\begin{equation*}
			\norm{T-T_n}{B(X,Y)}{}\leq\left(2\norm{T}{B(X_0,Y_0)}{}\right)^{1-\theta}\norm{T-T_n}{B(X_1,Y_1)}{\theta}\rightarrow{0}\quad\textnormal{as}\quad n\rightarrow\infty.
		\end{equation*}
		To show the second part of the theorem we define $F_n$ as follows 
		\begin{equation*}
			F_n(x)=\sum_{j=1}^{\sigma_n}\left(\frac{1}{\mu\left(E_j^{(n)}\right)}\int_{E_j^{(n)}}xd\mu\right)\chi_{E_j^{(n)}}
		\end{equation*} 
		for any $n\in\mathbb{N}$ and $x\in{X_0+X_1}$. Then, we proceed as previously to get
		\begin{equation*}
			\norm{T-T_n}{B(X,Y)}{}\rightarrow{0}\quad\textnormal{as}\quad n\rightarrow\infty.
		\end{equation*}
		Furthermore, $T_n=F_n\circ{T}$ is admissible with respect to the couples $(X_0,X_1)$ and $(Y_0,Y_1)$, i.e. for any $x\in{X_i}$ and $n\in\mathbb{N}$,
		\begin{align*}
			\norm{T_n(x)}{Y_i}{}=\norm{F_n\circ T(x)}{Y_i}{}\leq\norm{F_n\circ T}{B(X,Y)}{}\norm{x}{X_i}{}\leq\norm{T}{B(X,Y)}{}\norm{x}{X_i}{},
		\end{align*}
		and consequently
		\begin{align*}
			\norm{T_n}{\mathcal{A}}{}=\max_{i=0,1}\norm{F_n\circ T}{B(X_i,Y_i)}{}\leq\norm{T}{\mathcal{A}}{}.
		\end{align*}		
	\end{proof}
	
	\begin{remark}
		Now, taking an arbitrary Banach couple $(A_0,A_1)$ we recall a construction $$K_{\theta, q}(\overline{A})=\{a\in A_0+A_1:\Phi_{\theta, q}(K(t,a))<\infty\}$$ for $\theta\in[0,1]$ and $q\in[1,\infty]$, where 
		\begin{equation*}
		\Phi_{\theta, q}(f(t))=\left(\int_{0}^{\infty}(t^{-\theta}f(t))^q\frac{dt}{t}\right)^{1/q}\quad\textnormal{and}\quad K(t,a)=\inf_{a=a_0+a_1}\{\norm{a_0}{A_0}{}+t\norm{a_1}{A_1}{}\}.
		\end{equation*}
		It is well known that $K_{\theta,q}$ is an exact interpolation functor of exponent $\theta$, see \cite{BergLofs}.
	\end{remark}

	\begin{theorem}
		Let $\overline{X}=(X_0,X_1)$, $\overline{Y}=(Y_0,Y_1)$ be two Banach couples such that $X_i, Y_i\hookrightarrow{L^1+L^\infty}$ for each $i=0,1$. If $Y_1$ is lower locally uniformly $K$-monotone and $T:\overline{X}\rightarrow\overline{Y}$ is an admissible operator such that $T:{X_1}\rightarrow{Y_1}$ compactly, then $T:K_{\theta, q}(\overline{X})\rightarrow{K_{\theta, q}(\overline{Y})}$ is compact. 
	\end{theorem}
	
	\begin{proof}
		Let $\theta\in[0,1]$ and $g\in[1,\infty)$. Then, by Theorem 3.1.2 \cite{BergLofs} it follows that $K_{\theta, q}(\overline{X})$ and $K_{\theta, q}(\overline{Y})$ are exact interpolation functors of exponent $\theta$ on the category $\mathcal{N}$. Hence, by Theorem \ref{thm:admissible:compact}, since $Y_1$ is lower locally uniformly $K$-monotone this yields that any admissible operator $T:K_{\theta, q}(\overline{X})\rightarrow{K}_{\theta, q}(\overline{Y})$ is compact.
	\end{proof}

\subsection*{Acknowledgement}
	The first author (Maciej Ciesielski) is supported by Poznan University of Technology,
	Poland, Grant No. 0213/SBAD/0118.

$\begin{array}{l}
\textnormal{\small Maciej CIESIELSKI}\\
\textnormal{\small Institute of Mathematics}\\
\textnormal{\small Pozna\'{n} University of Technology}\\ 
\textnormal{\small Piotrowo 3A, 60-965 Pozna\'{n}, Poland}\\ 
\textnormal{\small email: maciej.ciesielski@put.poznan.pl;}\\
\textnormal{ }
\end{array}$

$\begin{array}{l}
\textnormal{\small Grzegorz LEWICKI}\\
\textnormal{\small Department of Mathematics and Computer Science}\\ \textnormal{\small Jagiellonian University}\\
\textnormal{\small 30-348 Krak\'ow, \L ojasiewicza 6, Poland}\\
\textnormal{\small email: grzegorz.lewicki@im.uj.edu.pl;}
\end{array}$

\end{document}